\documentclass[runningheads]{llncs}
\usepackage{url}
\usepackage{multirow}
\usepackage{amssymb}
\usepackage{amsxtra}
\usepackage{makeidx}
\usepackage{amsfonts}
\usepackage{mathtools}
\usepackage{algorithm}
\usepackage{algorithmic}
\usepackage{float}
\usepackage{comment}
\usepackage[font=small,labelfont=bf]{caption}
\usepackage{breqn}
\usepackage[english]{babel}

\usepackage{subfigure}

\usepackage{color}

\title{Sequential Subspace Optimization for Quasar-Convex Optimization Problems with Inexact Gradient\thanks{ The research was supported by Russian Science Foundation (project No. 21-71- 30005).}}
\titlerunning{Sequential Subspace Optimization}
 
    \author{Ilya A. Kuruzov \inst{1,3}\orcidID{0000-0002-2715-5489} \and Fedor S. Stonyakin  \inst{1,2} \orcidID{0000-0002-9250-4438}}

\authorrunning{I. Kuruzov et al.}

\institute{Moscow Institute of Physics and Technology, Moscow, Russia
	\and
	V.\,I.\,Vernadsky Crimean Federal University, Simferopol, Russia\\
	\and
	Institute for Information Transmission Problems RAS, Moscow, Russia
	\email{kuruzov.ia@phystech.edu,  fedyor@mail.ru}
	}

\begin{document}

\maketitle

\begin{abstract}
It is well-known that accelerated gradient first-order methods possess optimal complexity estimates for the class of convex smooth minimization problems. In many practical situations it makes sense to work with inexact gradient information. However, this can lead to an accumulation of corresponding inexactness in the theoretical estimates of the rate of convergence. We propose one modification of the Sequential Subspace Optimization Method (SESOP) for minimization problems with $\gamma$-quasar-convex functions with inexact gradient. A theoretical result is obtained indicating the absence of accumulation of gradient inexactness. A numerical implementation of the proposed version of the SESOP method and its comparison with the known Similar Triangle Method with an inexact gradient is carried out.

\end{abstract}

\section*{Introduction}\label{sec1_introduction}

It is well-known that accelerated gradient-type methods possess optimal complexity estimates \cite{NemirovskyYudin} for the class of convex smooth minimization problems. In many practical situations it makes sense to work with inexact gradient information (see e.g. \cite{DevolderThesis}, \cite{DevPaper}, \cite{Polyak1987}, \cite{Vasin2021}). For example, this is relevant for gradient-free optimization methods (when estimating the gradient by finite differences) in infinite dimensional spaces for inverse problems (see, e.g. \cite{Kabanikhin}).

However, this can lead to an accumulation of corresponding inexactness in the theoretical estimates of the rate of convergence. Let us consider minimization problems of convex and $L$-smooth function $f$ ($\|\cdot\|$ is a usual Euclidean norm) 
\begin{equation}\label{Lsmothness}
\|\nabla f(x) - \nabla f(y)\| \leqslant L \|x - y\| \quad \forall x, y \in \mathbb{R}^n    
\end{equation}
with an inexact gradient $g: \mathbb{R}^n \rightarrow \mathbb{R}^n$:
\begin{equation}\label{InexactGrad}
\|g(x) - \nabla f(x)\| \leqslant \delta,
\end{equation}
where $L>0$ and $\delta > 0$. For the considered class of problem, the following estimate for accelerated gradient-type methods:
$$
f(x_N) - \min_{x \in \mathbb{R}^n} f(x) = O\left(L\|x_0 - x^*\|^2N^{-2} + \delta\max_{k \leq N }{\|x_k - x^*\|}\right)
$$
is known \cite{Aspremont2008,Vasin2021} for each  $x^*:$ $f(x^*) = \min_{x \in \mathbb{R}^n} f(x)$. It is clear that the quantity $\max_{k \leq N }{\|x_k - x^*\|}$ can be not small enough. In this paper, we propose one modification of the Sequential Subspace Optimization Method \cite{SESOP_2005} with an $\delta$-additive noise in gradient \eqref{InexactGrad} and prove 
the following estimate:
$$f(x_N) - \min_{x \in \mathbb{R}^n} f(x) =  O\left(L\|x_0 - x^*\|^2N^{-2} + \delta \|x_0 - x^*\|\right),$$
where $x_0$ is the starting point of algorithm. Thus, a certain solution to the problem of accumulating the gradient inexactness is proposed for a special accelerated gradient method. It is important that we also consider some type of non-convex problems \cite{Hardt,NearOpt}.

The article consists of an introduction, 3 main sections and conclusion.

In the first main section \ref{sect1}, we propose and analyze a new modification of Subspace Optimization Method (Algorithm \ref{alg:sesop}) for minimization problems with $\gamma$-quasar-convex functions with an inexact gradient. The use of such a specific method made it possible to obtain a significant result on the non-accumulation of the additive gradient inexactness in the estimate of the convergence rate (Theorem \ref{SESOP_theorem}). 

However, this result for Algorithm \ref{alg:sesop} is essentially tied to the structure of this method, which is associated with auxiliary low-dimensional minimization problems. Therefore, it is important to investigate the influence of errors in solving such problems on the final estimate of the rate of convergence. Section \ref{sect2} is devoted to this question and Theorem \ref{SESOP_theorem_full} is obtained. 

The last main section \ref{sect3} is devoted to numerical illustration of the obtained theoretical results for one example of a quadratic function minimization problem. Firstly, we show that the convergence may be significantly better than the theoretical estimates for Algorithm \ref{alg:sesop}. Secondly, we compare Algorithm \ref{alg:sesop} with another known accelerated Similar Triangles Method (STM) for the case of additive gradient noise \cite{Vasin2021}. The STM was chosen for comparison with Algorithm \ref{alg:sesop} for the following reasons:
\begin{itemize}
\item in the case of exact gradient information ($\delta = 0$ in \eqref{InexactGrad}) both the STM and the SESOP possess the optimal rate of convergence $O(N^{-2})$; 
\item for STM with an inexact gradient \eqref{InexactGrad}, a theoretical estimate of the quality of the solution is known (\cite{Vasin2021}, Theorem 1 and Remark 3).
\end{itemize}

Let us introduce some auxiliary notations and definitions. 

Throughout this paper $\langle \mathbf{x}, \mathbf{y} \rangle$ $\|\cdot\|$ means the inner product of vectors $\mathbf{x} = (x_1, x_2, ..., x_n), \, \mathbf{y} = (y_1, y_2, ..., y_n) \in\mathbb{R}^n$ and is given by the formula 
$\langle \mathbf{x}, \mathbf{y} \rangle = \sum\limits_{k=1}^n x_ky_k$.

It turns out that it is possible to formulate the main results of the work for a certain class of not necessarily convex problems. Let us recall the definition of the class of $\gamma$-quasar-convex functions (see \cite{Hardt,NearOpt}).
\begin{definition}
Assume that $\gamma \in (0,1]$ and let $\mathbf{x}^*$ be a minimizer of the differentiable function $f:\mathbb{R}^n \rightarrow \mathbb{R}$. The function $f$ is $\gamma$-quasar-convex with respect to $\mathbf{x}^*$ if for all $\mathbf{x}\in\mathbb{R}^n$,
\begin{equation}
    f(\mathbf{x}^*)\geq f(\mathbf{x}) + \frac{1}{\gamma}\langle\nabla f(\mathbf{x}), \mathbf{x}^*-\mathbf{x}\rangle.
\end{equation}


\end{definition}

For example, a non-convex function $f(x) = |x| (1 - e^{-|x|})$ is a $1$-quasar-convex \cite{AccWQC,NearOpt}. The class of $\gamma$-quasar-convex functions is also called $\gamma$-weakly quasi-convex functions (see \cite{AccWQC}). Clearly each convex function is also $1$-quasar-convex. So, all results of this paper are applicable to convex optimization problems with an inexact gradient information.





\section{Subspace Optimization Method with Inexact Gradient}\label{sect1}

In this section we present some variant of the SESOP (Sequential Subspace Optimization) method \cite{SESOP_2005} for $\gamma$-quasar-convex functions with an inexact gradient. We generalize the results \cite{AccWQC} for SESOP method with inexact gradient on the class of $\gamma$-quasar-convex functions. In other words, our modifications of the SESOP method works with some approximation $g(x)$ of gradient $\nabla f(x)$ at each point $x \in \mathbb{R}^n$. 

Similarly to \cite{AccWQC,SESOP_2005} we start with a description of the investigated algorithm. Let $D_k=\|\mathbf{d}_k^0\,\mathbf{d}_k^1\,\mathbf{d}_k^2\|$ be an $n\times 3$ matrix, the columns of which are the following vector:
$$\mathbf{d}_k^0 = g(\mathbf{x}_k), \quad \mathbf{d}_k^1 = \mathbf{x}_k-\mathbf{x}_0,\quad\mathbf{d}_k^2 = \sum\limits_{i=0}^k \omega_i g(\mathbf{x}_i),$$
where $\omega_0=1, \omega_i = \frac{1}{2}+\sqrt{\frac{1}{4}+\omega_{i-1}^2}.$

For all $k\geq 1$ we have $\omega_k=\frac{1}{2}+\sqrt{\frac{1}{4}+\omega_{k-1}^2} \geq \frac{1}{2} + \omega_{k-1}$
and $\omega_k=\frac{1}{2}+\sqrt{\frac{1}{4}+\omega_{k-1}^2} \leq 1 + \omega_{k-1}.$

So, it holds that
\begin{equation}
    \label{w_k_est}
    k+1\geq \omega_k\geq \frac{k+1}{2}.
\end{equation}

The matrices $D_k$ will generate the subspaces over which we will minimize our objective function. With $D_k$ defined this way, the proposed algorithm takes the following form:

\begin{algorithm}
	\caption{A modification of the SESOP method with an inexact gradient}
  \label{alg:sesop}
  \begin{algorithmic}[1]
  \REQUIRE objective function $f$ with an inexact gradient $g$, initial point $\mathbf{x}_0$, number of iterations $T$.

  \FOR{$k =0, \hdots, T-1$}
      \STATE Find the optimal step
      \begin{equation}\label{eqsubproblem}
          \tau_k \leftarrow \arg\min_{\tau\in\mathbb{R}^3}f\left(\mathbf{x}_k + D_k\tau\right)
      \end{equation}
      \STATE $\mathbf{x}_{k+1} \leftarrow \mathbf{x}_k + D_k\tau_k$
  \ENDFOR
\RETURN $\mathbf{x}_T$
	\end{algorithmic}
\end{algorithm}

Let us show that the main advantage of the SESOP method with an inexact gradient is the absence in theoretical estimate the term $\max_k \|x_k - x^*\|$. Before the proof of this result, we need to estimate the error accumulation in $\mathbf{d}_k^3$ vector.

\begin{lemma}
\label{a_lemma_w}
Let the objective function $f$ be $L$-smooth and $\gamma$-quasar-convex with respect to $\mathbf{x}^*$. Suppose also for the inexact gradient $g:\mathbb{R}^n\rightarrow\mathbb{R}^n$ of function $f$ there is some constant $\delta_1\geq 0$ such that for all $\mathbf{x}\in\mathbb{R}^n$
    \begin{equation}
        \|g(\mathbf{x})-\nabla f(\mathbf{x})\|\leq\delta_1.
    \end{equation}
Let $\{\mathbf{x}_j\}_j$ be a sequence of points generated by Algorithm \ref{alg:sesop}. Then the following inequality holds:
\begin{equation}
    \left\|\sum\limits_{k=0}^{T} \omega_k g(\mathbf{x}_k)\right\|^2 \leq 2\sum\limits_{j=0}^T \omega_j^2 \|g(\mathbf{x}_j)\|^2 + 72 T^4 \delta_1^2
\end{equation}
for each $T\in\mathbb{N}$.
\end{lemma}

\begin{proof}
Let us define $W_{T} = \left\|\sum\limits_{k=0}^{T} \omega_k g(\mathbf{x}_k)\right\|$. Note that  $\mathbf{d}^2_j=\sum\limits_{k=0}^{j} \omega_k g(\mathbf{x}_k)$ and $\|d_j^2\| \leq W_j$ for each $j$. For $W_{T}$ we have equality for this value:
\begin{equation}
\label{W_T_2_expression}
    W_T^2 = \omega_T^2 \|g(\mathbf{x}_T)\|^2 + 2\left\langle w_T g(\mathbf{x}_T), \mathbf{d}^2_{T-1}\right\rangle + W_{T-1}^2
\end{equation}
and we have that $$\nabla f(\mathbf{x}_T) \perp  \mathbf{d}^2_{T-1},$$ because of optimizing on the subspace $\mathbf{x}_{T-1}~+~ \mathbf{d}^2_{j-1}$. Therefore we have the following inequality $\left(\left\langle g(\mathbf{x}_j), \ \mathbf{d}^2_{j-1}\right\rangle = 0\right)$:
\begin{equation}
    \label{delta_ineq}
    |\left\langle g(\mathbf{x}_j), \ \mathbf{d}^2_{j-1}\right\rangle| = |\left\langle 
    \nabla f(\mathbf{x}_j),  \mathbf{d}^2_{j-1}\right\rangle + \left\langle g(\mathbf{x}_j) - \nabla f(\mathbf{x}_j),  \mathbf{d}^2_{j-1}\right\rangle| \leq \delta_1 W_{j-1}
\end{equation}
for all $j\geq 1.$ So, we have the following correlations for $W_T$:
$$W_T^2 = \sum\limits_{j=0}^T \omega_j^2 \|g(\mathbf{x}_j)\|^2 + 2\sum\limits_{j=1}^T w_j\left\langle  g(\mathbf{x}_j), \mathbf{d}^2_{j-1}\right\rangle,$$
and
\begin{equation}
    \label{up_W_est}
    W_T^2 \leq \sum\limits_{j=0}^T \omega_j^2 \|g(\mathbf{x}_j)\|^2 +2\delta_1 \sum\limits_{j=1}^T w_j W_{j-1}.
\end{equation}

On the other hand, we can estimate the inner product by the Cauchy–Schwarz inequality  for $T=j$ in \eqref{W_T_2_expression} and to get the following estimate:
$$ W_j^2 \geq - 2w_j \delta_1 W_{j-1} + W_{j-1}^2.$$

Solving the previous inequality on $W_j$, we have ($\sqrt{a + b} \leq \sqrt{a} + \sqrt{b}$ for all $a, b \geq 0$)
$$W_{j-1} \leq w_j\delta_1 + \sqrt{w_j^2 \delta_1^2 + W_j^2} \leq 2w_j\delta_1 + W_j.$$

By induction we have the following estimate:
$$W_j \leq 2\delta_1 \sum\limits_{k=j+1}^T w_k + W_T\text{ for } j=\overline{0, T-1}.$$
From \eqref{w_k_j_T} we have that
$$W_j \leq \delta_1 T(T+3) + W_T\text{ for } j=\overline{0, T-1}.$$
The last inequality and \eqref{up_W_est} mean that
\begin{equation}
    W_T^2 \leq \sum\limits_{j=0}^T \omega_j^2 \|g(\mathbf{x}_j)\|^2 +2\delta_1^2 T (T+3) +  2\delta_1W_T  \sum\limits_{j=1}^T w_j, 
\end{equation}

\begin{equation}
    W_T^2 \leq  \sum\limits_{j=0}^T \omega_j^2 \|g(\mathbf{x}_j)\|^2 +2\delta_1^2 T (T+3) + T(T + 3) \delta_1 W_T.
\end{equation}

The value $W_T \geq 0$ by definition. One of the roots of the previous quadratic inequality is always negative. So, for $W_T$ to meet this inequality, its value must be not more than the largest root of the corresponding quadratic function:
$$
    W_T \leq \frac{1}{2}(T^2 + 3T)\delta_1 + \frac{1}{2}\sqrt{(T^2 + 3T)^2 \delta_1^2 + 4\sum\limits_{j=0}^T \omega_j^2 \|g(\mathbf{x}_j\|^2 +8\delta_1^2 T (T+3)}.$$

Taking into account $T\geq 1$, we have
$$W_T \leq 2T^2\delta_1 + \frac{1}{2}\sqrt{48 T^4 \delta_1^2 + 4\sum\limits_{j=0}^T \omega_j^2 \|g(\mathbf{x}_j)\|^2};$$

Due to the inequality $\sqrt{a+b}\leq\sqrt{a}+\sqrt{b}$ (for all $a, b \geq  0$)
we have
$$W_T \leq 6 T^2\delta_1 + \sqrt{\sum\limits_{j=0}^T \omega_j^2 \|g(\mathbf{x}_j)\|^2}.$$
Further, the inequality $(a+b)^2\leq 2a^2+2b^2$ (for all $a, b \in \mathbb{R}$) means that
$$W_T^2 \leq 72 T^4 \delta_1^2 + 2\sum\limits_{j=0}^T \omega_j^2 \|g(\mathbf{x}_j)\|^2.$$
\end{proof}

Using Lemma \ref{a_lemma_w}, we can prove the following main result of this section.

\begin{theorem}
\label{SESOP_theorem}
Let the objective function $f$ be $L$-smooth and $\gamma$-quasar-convex with respect to $\mathbf{x}^*$. Also, for the inexact gradient $g:\mathbb{R}^n\rightarrow\mathbb{R}^n$ there is some constant $\delta_1\geq 0$ such that for all $\mathbf{x}\in\mathbb{R}^n$
    \begin{equation}
    \label{g_cond_i}
        \|g(\mathbf{x})-\nabla f(\mathbf{x})\|\leq\delta_1.
    \end{equation}
Then the sequence $\{\mathbf{x}_k\}$ generated by Algorithm \ref{alg:sesop} satisfies
\begin{equation}
    \label{est_SESOP}
     f(\mathbf{x}_k)-f^*\leq \frac{8LR^2}{\gamma^2 k^2}  + 4\left(\frac{R}{\gamma} + 17 \right)\delta_1,
\end{equation}
where $R=\|\mathbf{x}^*-\mathbf{x}_0\|$.
\end{theorem}

\begin{proof}
By constructing $\mathbf{x}_{k+1}$ we have the following inequality:
\begin{equation}
    \label{k_min_ineq}
    f(\mathbf{x}_{k+1}) = \min_{\mathbf{s}\in\mathbb{R}^3} f\left(\mathbf{x}_k + \sum\limits_{i=0}^2 s_i d_k^i\right)\leq f\left(\mathbf{x}_k + s_0 g(\mathbf{x}_k)\right).
\end{equation}

On the other hand, from $L$-smoothness we have that
$$f(\mathbf{y})\leq f(\mathbf{x}) + \langle\nabla f(\mathbf{x}), \mathbf{y} - \mathbf{x}\rangle + \frac{L}{2}\|\mathbf{x}-\mathbf{y}\|^2$$
for all $\mathbf{x}, \mathbf{y}\in\mathbb{R}^n$. Further, from $\eqref{g_cond_i}$ follows the corresponding inequality for an inexact gradient:
$$f(\mathbf{y})\leq f(\mathbf{x}) + \langle g(\mathbf{x}), \mathbf{y} - \mathbf{x}\rangle + \frac{L}{2}\|\mathbf{x}-\mathbf{y}\|^2 + \delta_1 \|\mathbf{x}-\mathbf{y}\| \quad \forall \mathbf{x}, \mathbf{y}\in\mathbb{R}^n.$$

Using the inequality $\delta_1 \|\mathbf{x}-\mathbf{y}\| = \left(\sqrt{\frac{1}{L}}\delta_1\right) \left(\sqrt{\frac{L}{1}}\|\mathbf{x}-\mathbf{y}\|\right)\leq \frac{\delta_1^2}{2L} + \frac{L}{2}\|\mathbf{x}-\mathbf{y}\|^2$, we have
\begin{equation}
    \label{g_L_ineq}
f(\mathbf{y})\leq f(\mathbf{x}) + \langle g(\mathbf{x}), \mathbf{y} - \mathbf{x}\rangle + L\|\mathbf{x}-\mathbf{y}\|^2 + \frac{\delta_1^2}{2L} \quad \forall \mathbf{x}, \mathbf{y}\in\mathbb{R}^n.
\end{equation}

On the base of the last inequality and the right part of \eqref{k_min_ineq} for $\mathbf{y}:=\mathbf{x}_k + s_0 g(\mathbf{x}_k)$ and $\mathbf{x}=\mathbf{x}_k$ we can conclude that
\begin{equation}
    f(\mathbf{x}_{k+1}) \leq f(\mathbf{x}_k) +\left(s_0 + s_0^2 L\right)\|g(\mathbf{x}_k)\|^2 + \frac{1}{2L}\delta_1^2
\end{equation}
for each $s_0\in\mathbb{R}$. Further,
\begin{equation}\label{eqineqs0}
     -\left(s_0 + s_0^2 L\right) \|g(\mathbf{x}_k)\|^2 \leq f(\mathbf{x}_k) - f(\mathbf{x}_{k+1}) + \frac{1}{2L}\delta_1^2.
\end{equation}
Maximizing the left part of \eqref{eqineqs0} by $s_0$, we have the following estimate for the inexact gradient norm:
\begin{equation}
      \|g(\mathbf{x}_k)\|^2 \leq 4L(f(\mathbf{x}_k) - f(\mathbf{x}_{k+1})) + 2\delta_1^2.
\end{equation}

So, $\nabla f(\mathbf{x}_k)\perp  \mathbf{x}_k - \mathbf{x}_0$ for all $k>0$ because $\mathbf{x}_k$ is minimizer of $f$ on the subspace containing the directions $\mathbf{x}_k-\mathbf{x}_{k-1}$ and $\mathbf{x}_{k-1}-\mathbf{x}_0$. Because of it we can write
$$\langle \nabla f(\mathbf{x}_k), \mathbf{x}_k-\mathbf{x}^*\rangle = \langle \nabla f(\mathbf{x}_k), \mathbf{x}_0-\mathbf{x}^* \rangle.$$
From this equality and \eqref{g_cond_i} we have
$$f(\mathbf{x}_k)-f(\mathbf{x}^*)\leq \frac{1}{\gamma} \langle g(\mathbf{x}_k), \mathbf{x}_0-\mathbf{x}^* \rangle + \delta_1 \frac{R}{\gamma}.$$

Similarly to the proof of Theorem 3.1 from \cite{AccWQC} we have the following chain of correlations:
\begin{equation}
\label{eps_1}
    \begin{aligned}
        \sum\limits_{k=0}^{T-1}\omega_k(f(\mathbf{x}_k)-f^*)&\leq \frac{1}{\gamma} \left\langle  \sum\limits_{k=0}^{T-1} \omega_k g(\mathbf{x}_k), \mathbf{x}_0-\mathbf{x}^* \right\rangle +\delta_1 \frac{R}{\gamma} \sum\limits_{k=0}^{T-1} \omega_k\\
        &\leq \frac{1}{\gamma} \left\|\sum\limits_{k=0}^{T-1} \omega_k g(\mathbf{x}_k)\right\|R +\delta_1 \frac{R}{\gamma} \sum\limits_{k=0}^{T-1} \omega_k.
    \end{aligned}
\end{equation}

Now we can estimate the multiplier $\left\|\sum\limits_{k=0}^{T-1} \omega_k g(\mathbf{x}_k)\right\|$. According to Lemma \ref{a_lemma_w} and \eqref{w_k_est}, we have the following estimates:
\begin{align*}
    \left\|\sum\limits_{k=0}^{T-1} \omega_k g(\mathbf{x}_k)\right\|^2 &\leq 2\sum\limits_{j=0}^{T-1} \omega_j^2 \|g(\mathbf{x}_j)\|^2 + 256 T^4 \delta_1^2 \\
    & \leq 8L \sum\limits_{k=0}^{T-1} \omega_k^2  (f(\mathbf{x}_k)-f(\mathbf{x}_{k+1})) + 260 T^4 \delta_1^2.
\end{align*}

Note that the choice of $\omega_k$ is equivalent to choosing the largest $\omega_k$ satisfying
$$\omega_k = \begin{cases}1, \text{if $k=0$},\\ \omega_k^2-\omega_{k-1}^2, \text{otherwise}.\end{cases}$$

Now we can estimate the left part of \eqref{eps_1} denoting $\varepsilon_k=f(\mathbf{x}_k)-f^*$ in the following way ($\sqrt{a+b} \leq \sqrt{a} + \sqrt{b}$ for all $a, b \geq 0$):

\begin{equation}
    \begin{aligned}
        S &= \sum\limits_{k=0}^{T-1}\omega_k \varepsilon_k \\
        &\leq \frac{1}{\gamma} \left\|\sum\limits_{k=0}^{T-1} \omega_k g(\mathbf{x}_k)\right\|R +\delta_1 \frac{R}{\gamma} \sum\limits_{k=0}^{T-1} \omega_k\\
         &\leq\left(\frac{8LR^2}{\gamma^2} \sum\limits_{k=0}^{T-1} \omega_k^2 (\varepsilon_k-\varepsilon_{k+1}) + 260 T^4\delta_1^2\right)^{\frac{1}{2}} +\delta_1 \frac{R}{\gamma} \sum\limits_{k=0}^{T-1} \omega_k\\
      &\leq\left(\frac{8LR^2}{\gamma^2} \sum\limits_{k=0}^{T-1} \omega_k^2 (\varepsilon_k-\varepsilon_{k+1})\right)^{\frac{1}{2}} + 17 T^2 \delta_1 +\delta_1 \frac{R}{\gamma} \sum\limits_{k=0}^{T-1} \omega_k\\
         &=\sqrt{\frac{8LR^2}{\gamma^2}} \sqrt{S - \varepsilon_T\omega_{T-1}^2} +  17 T^2 \delta_1  +\delta_1 \frac{R}{\gamma} \sum\limits_{k=0}^{T-1} \omega_k.
    \end{aligned}
\end{equation}

From the inequality above we have
\begin{equation}\label{eqat26}
    \omega_{T-1}^2 \varepsilon_T \leq S - \frac{\gamma^2}{8LR^2}\left(S - \delta_1 \frac{R}{\gamma} \sum\limits_{k=0}^{T-1} \omega_k - 17 T^2 \delta_1 \right)^2.
\end{equation}

Maximizing the right part of \eqref{eqat26} on $S$ we get
\begin{equation}
    \omega_{T-1}^2 \varepsilon_T \leq \frac{2LR^2}{\gamma^2} +  \delta_1 \frac{R}{\gamma} \sum\limits_{k=0}^{T-1} \omega_k + 17 T^2 \delta_1.
\end{equation}

Now from \eqref{w_k_est} we have
\begin{equation}
    \begin{aligned}
    \omega_{T-1}^2 \varepsilon_T &\leq \frac{2LR^2}{\gamma^2} +  \delta_1 \frac{R}{\gamma} \sum\limits_{k=0}^{T-1} \omega_k +  17 T^2 \delta_1\\
    &\leq\frac{2LR^2}{\gamma^2} +  T^2\left( \delta_1 \frac{R}{\gamma} +17 \delta_1\right).
    \end{aligned}
\end{equation}

Dividing both parts of this inequality by $w_{T-1}^2$  and using the lower estimate for it \eqref{w_k_est} we get

\begin{equation}
    \begin{aligned}
 \varepsilon_T \leq \frac{8LR^2}{\gamma^2 T^2} + 4\delta_1 \frac{R}{\gamma} + 68\delta_1 =
 \frac{8LR^2}{\gamma^2 T^2} + 4\left(\frac{R}{\gamma} + 17 \right)\delta_1,
    \end{aligned}
\end{equation}
Q.E.D.
\end{proof}

\section{Subspace Optimization Method with Inexact Solutions of Auxiliary Subproblems}\label{sect2}

The result of the previous section shows that the SESOP algorithm can work with additive noise in a gradient. It is essential that the method leads to the need to solve auxiliary low-dimensional optimization problems.
So, there is an interesting case when the auxiliary problem \eqref{eqsubproblem} cannot be solved exactly. We consider this case in the following theorem.

\begin{theorem}
\label{SESOP_theorem_full}
Let the objective function $f$ be $L$-smooth and $\gamma$-quasar-convex with respect to $\mathbf{x}^*$. Let $\tau_k$ be the step value obtained with the inexact solution of the auxiliary problem \eqref{eqsubproblem} on step 2 in Algorithm \ref{alg:sesop} on the $k$-th iteration. Namely, the following conditions for inexactness hold:
\begin{itemize}
    \item[(i)] For the inexact gradient $g:\mathbb{R}^n\rightarrow\mathbb{R}^n$ there is some constant $\delta_1\geq 0$ such that for all points $\mathbf{x}\in\mathbb{R}^n$ condition \eqref{g_cond_i} holds.
    \item[(ii)] The inexact solution $\tau_k$ meets the following condition:
    \begin{equation}
\label{grad_orth_i}
|\left\langle \nabla f(\mathbf{x}_k), \mathbf{d}_{k-1}^2 \right\rangle |\leq k^2 \delta_2
\end{equation}
for some constant $\delta_2 \geq 0$ and each $k\in\mathbb{N}$. Note that $\mathbf{x}_k=\mathbf{x}_{k-1}+D_{k-1}\tau_{k-1}$.
    \item[(iii)] The inexact solution $\tau_k$ meets the following condition for some constant $\delta_3 \geq 0$:
    \begin{equation}
    \label{grad_orth_ii}
        |\left\langle \nabla f(\mathbf{x}_k), \mathbf{x}_k - \mathbf{x}_0\right\rangle| \leq \delta_3.
    \end{equation}
    \item[(iv)] The problem from step 2 in Algorithm \ref{alg:sesop} is solved with accuracy $\delta_4 \geq 0$ on function on each iteration, i.e.
    $f(\mathbf{x}_k) - \min_{\tau\in\mathbb{R}^n}f(\mathbf{x}_{k-1} + D_{k-1}\tau) \leq \delta_4$.
\end{itemize}
Then the sequence $\{\mathbf{x}_k\}$ generated by Algorithm \ref{alg:sesop} satisfies
\begin{equation}
 f(\mathbf{x}_k)-f^* \leq \frac{8LR^2}{\gamma^2 k^2} + \left(\frac{R}{\gamma}+ 10\right)\delta_1 + 4\sqrt{\delta_2} +  \delta_3 + 5\sqrt{\frac{L\delta_4}{k}}
\end{equation}
for each $k\geq 8$, where $R=\|\mathbf{x}^*-\mathbf{x}_0\|$.
\end{theorem}

\begin{proof}
The proof of this theorem is somewhat similar to the proof of Theorem \ref{SESOP_theorem} and was moved to the appendix C.
\end{proof}

\begin{remark}
The obtained estimate of the rate of convergence for Algorithm \ref{alg:sesop} does not depend on the value $\max_k\|\mathbf{x}_k-\mathbf{x}^*\|$ and it depends only on $R, L$ and $\gamma > 0$. 
\end{remark}

\begin{remark}
It is clear that when the auxiliary problem \eqref{eqsubproblem} in Algorithm \ref{alg:sesop} has exact solution then $\delta_2=\delta_3=\delta_4=0$. The constant before $\delta_1$ was improved in comparison with the result from Theorem \ref{SESOP_theorem} because of more accurate work with constants in proofs (see Lemmas \ref{a_lemma_w} and \ref{a_lemma_w_full}).
\end{remark}

According to Theorem \ref{SESOP_theorem_full}, the SESOP method for a $\gamma$-convex function can find solution with quality $\varepsilon$ by function after $N=\sqrt{\frac{16LR^2}{\gamma^2 \varepsilon}}$ iterations when the following condition holds:
$$\left(\frac{R}{\gamma}+ 10\right)\delta_1 + 4\sqrt{\delta_2} +  \delta_3 + 5\sqrt{\frac{L\delta_4}{k}}\leq\frac{\varepsilon}{2}.$$
In particular, the SESOP method finds solution with this quality  after $N=\sqrt{\frac{16LR^2}{\varepsilon}}$ iterations for convex functions.

Now we want to discuss the relationship between conditions $(ii)$, $(iii)$ and $(iv)$ from Theorem \ref{SESOP_theorem_full}. The condition on the accuracy of the subproblem $(iv)$ is  natural enough for such methods. Conditions $(ii)$ and $(iii)$ are caused by the form of the method and provide almost orthogonality of the gradient and vectors $\mathbf{d}_k^j, j=1,2$. We can prove the following simple result.

\begin{theorem}
\label{delta_link}
If condition $(iv)$ from Theorem \ref{SESOP_theorem_full} holds, then we can choose $\delta_2, \delta_3 \geq 0$ according to the following estimates:
$$\delta_3 \leq \sqrt{2L \delta_4}\left(\sqrt{\max_{k}(\|D_k\| \|\tau_k\|)} + \sqrt{\|\max_{k}\mathbf{d}^1_{k-1}\|}\right)$$
and
$$\delta_2 \leq \frac{1}{k^2}\sqrt{2L \max_{k}\|\mathbf{d}_k^3\|\delta_4}.$$
\end{theorem}
\begin{proof}
Now we want to express conditions \eqref{grad_orth_i} and \eqref{grad_orth_ii} through accuracy of subproblem solution \eqref{eqsubproblem} $\delta_4$. We need to introduce the following auxiliary function:
\begin{equation}
\label{f_k}
    f_k(\tau) = f(\mathbf{x}_k + D_k \tau).
\end{equation}

Note that $f_k:  \mathbb{R} \rightarrow \mathbb{R}$ and its gradient is a one-dimension derivative. Let function $f_k$ have a Lipschitz continuous gradient with constants $L_k^j$, $j={1,3}$. We can derive these constants from $L$ and the norms of directions $d_k^j$:

$$\left|\frac{d}{d\tau_j} f_k(\tau+\alpha e_j)- \frac{d}{d\tau_j} f_k(\tau)\right| = \left|\left\langle d_k^j, \nabla f(\mathbf{x}^1) - \nabla f(\mathbf{x}^2)\right\rangle\right| \leq L \|d_k^j\| |\alpha|.$$
where $e_j\in\mathbb{R}^3$ is the $j$-th vector in the standard basis, $\alpha\in\mathbb{R}$ is some constant, $\mathbf{x}^1=\mathbf{x}_k + D_k \tau$ and $\mathbf{x}^2=\mathbf{x}_k + D_k (\tau+\alpha e_j)$. So we have the following expression for Lipschitz constant of a gradient for $f_k$ with respect to the $j$-th component: 
\begin{equation}
\label{L_k_j}
    L_k^j = L\|d_k^j\|.
\end{equation}
It is easy to see that
$$\left|\frac{d}{d\tau_j} f_k(\tau)\right|^2 \leq 2L_k^j \left(f_k(\tau) - \min_{\tau_j} f_k(\tau)\right)\leq  2L_k^j \left(f_k(\tau) - \min_{\tau} f_k(\tau)\right)=2L_k^j \delta_4$$
for all $\tau_j$. From \eqref{L_k_j}, the inequality above and the definition of $f_k$ \eqref{f_k}, we have the following expression:
\begin{equation}
\label{g_orth_delta_4}
    |\langle\nabla f(\mathbf{x}_{k+1}), d_k^j\rangle|\leq \sqrt{2L \|\mathbf{d}_k^j\| \delta_4}.
\end{equation}

It means if we choose $\delta_2>0$ in the following way:
$$\delta_2 \leq \frac{1}{k^2}\sqrt{2L \max_{k}\|\mathbf{d}_k^3\|\delta_4}.$$
then condition $(ii)$ in Theorem \ref{SESOP_theorem_full} meets.

In a similar way we can obtain that $f_k$ has Lipschitz continuous gradient with constant $L_k$:
$$L_k = \|D_k\| L$$
and
$$\left|\nabla_\tau f_k(\tau)\right|^2 \leq  2L_k \left(f_k(\tau) - \min_{\tau} f_k(\tau)\right)=2L_k \delta_4.$$
Note that
$$\mathbf{x}_{k}-\mathbf{x}_0 = D_{k-1} \tau_{k-1} + \mathbf{d}^1_{k-1} $$
Finally, we can choose $\delta_3$ in the following way:
$$\delta_3 \leq \sqrt{2L \delta_4}\left(\sqrt{\max_{k}(\|D_k\| \|\tau_k\|)} + \sqrt{\|\max_{k}\mathbf{d}^1_{k-1}\|}\right).$$
\end{proof}

\section{Numerical Experiments}\label{sect3}

In the current section we provide the results of numerical experiments. All experiments were carried out on Python 3.7.3 on computer Acer Swift 5 SF514-55TA-56B6 with processor Intel(R) Core(TM) i5-8250U @ CPU 1.60GHz, 1800 MHz.

All experiments were carried out in the assumption that we can solve the subspace optimization problem at each iteration with some accuracy on function. For this, we used the quadratic test function
$$f(\mathbf{x})= \mathbf{x}^\top A \mathbf{x} + 2\mathbf{b}^\top x$$
with $A\in \mathbb{S}^n_+$ ($A$ is a symmetric positive semidefinite matrix), $\mathbf{b}\in\mathbf{R}^n$. Obviously, this function is convex and consequently $1$-quasar-convex. The components of parameter $\mathbf{b}$ were generated randomly i.i.d. from uniform distribituion $\mathcal{U}([-1,1])$. The matrix $A=B^\top B$ where components $B\in\mathbb{R}^{n\times n}$ were generated by the same way as for vector $\mathbf{b}$.

The shift $\tau_k$ can be found as a solution of convex quadratic optimization problem:
$$\min_{\tau\in\mathbf{R}^3} \tau^\top D_k^\top A D_k\tau - 2 \left( \mathbf{b} + A \mathbf{x}_k\right)^\top D_k \tau $$
with any accuracy that we will vary in our experiments (see details below). The Lipschitz constant $L$ of $\nabla f$ is also known and equals the maximal singular value of matrix $A$. For all experiments we take dimension $n=500$.

\begin{figure}[ht!]  
\vspace{-4ex} \centering \subfigure{
\includegraphics[width=0.45\linewidth]{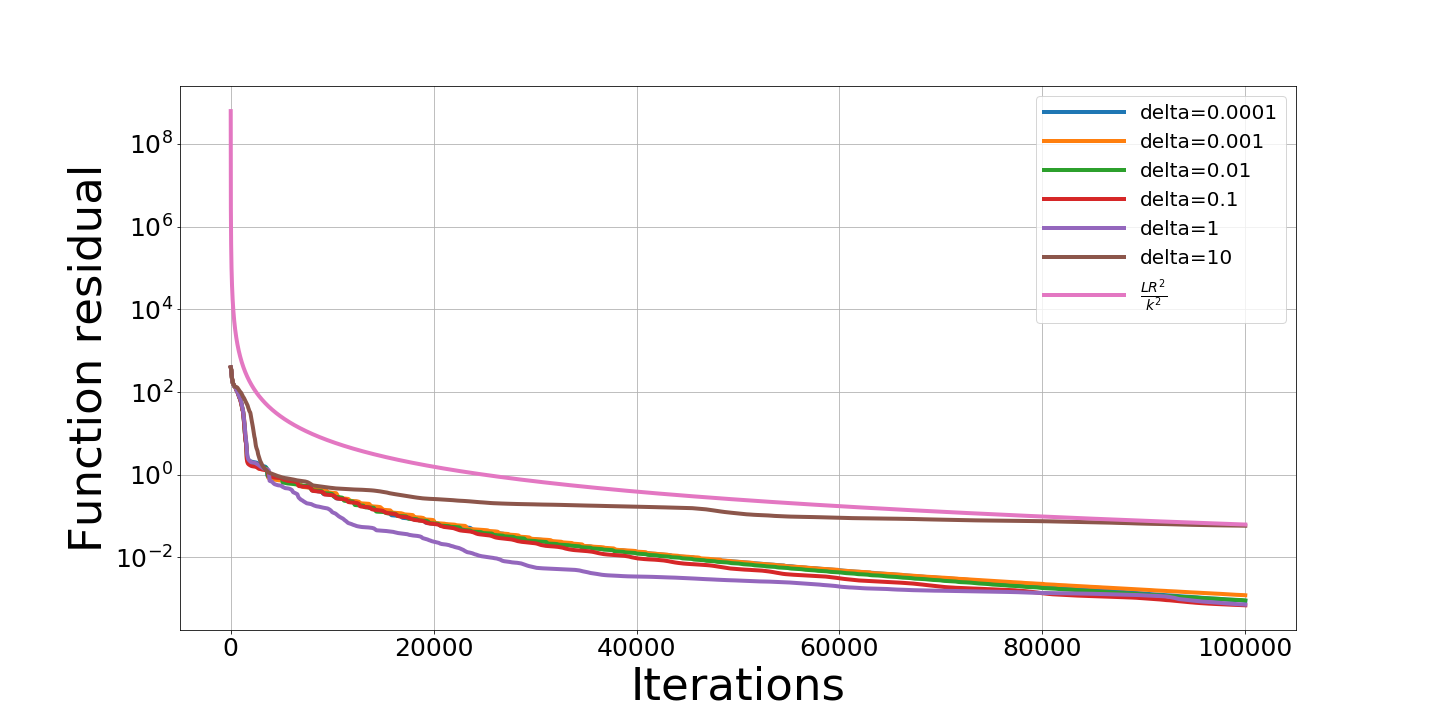} \label{fig:grad_exp_f_N} }  
\hspace{4ex}
\subfigure{
\includegraphics[width=0.45\linewidth]{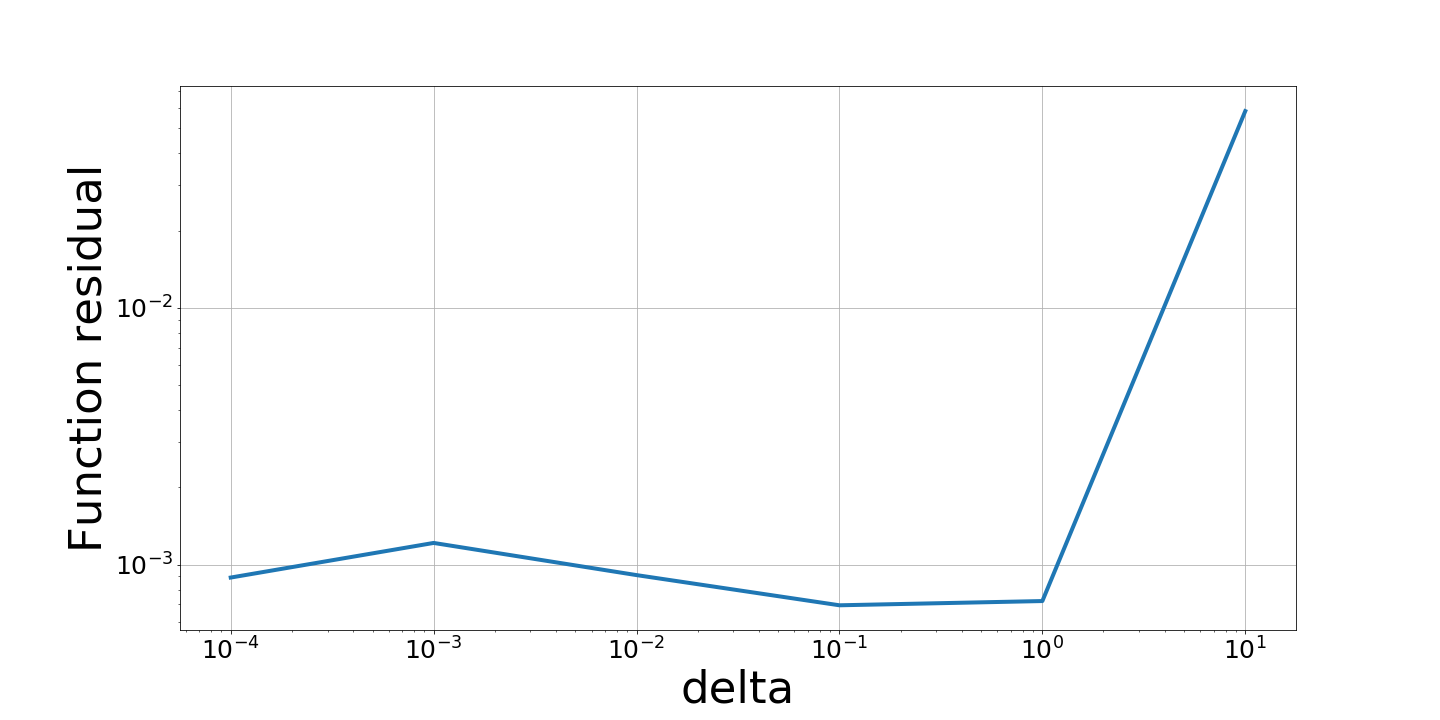} \label{fig:grad_exp_f_delta} }
\caption{The dependencies of convergence on gradient inexactness in the case of an exact solution of the subspace optimization problem: \subref{fig:grad_exp_f_N} convergence for different $\delta_1$; \subref{fig:grad_exp_f_delta}  minimal values  found for different $\delta_1$.} \label{fig:res_1}
\end{figure}

The first experiment compares the theoretical estimation from Theorem \ref{SESOP_theorem} and the real experiment in the case of inexactness in the gradient only.  It means that we solve quadratic optimization problem with machine accuracy that is significantly less than inexactness in the gradient. The inexact gradient will be given as an usual gradient with some noise
$g(\mathbf{x})= \nabla f(\mathbf{x}) + \delta \xi(\mathbf{x}),$
where $\xi(\mathbf{x})\sim\mathcal{U}\left(S_1(0)\right)$ is a random vector from the unit sphere with uniform distribution. Obviously such a vector meets the conditions of Theorem \ref{SESOP_theorem}. The results of this experiment are presented in Figure \ref{fig:res_1}.

We can see in Figure \ref{fig:grad_exp_f_delta} that the convergence of the proposed variant of the SESOP method (Algorithm \ref{alg:sesop}) at the first 100000 iterations is better than the theoretical convergence (the line $\frac{LR^2}{k^2}$ on graph) without noise for any gradient inexactness for $\delta\in\left[10^{-4}, 10\right]$. Moreover, in Figure \ref{fig:grad_exp_f_delta} the dependence of the function residual on the gradient inexactness shows that there is no significant error accumulation for $\delta<1$ at the first 100000 iterations. Such an optimistic result was obtained by Algorithm \ref{alg:sesop} due to the exact solution of the low-dimensional optimization subproblems \eqref{eqsubproblem}.

\begin{figure}[ht!]  
\vspace{-4ex} \centering \subfigure{
\includegraphics[width=0.45\linewidth]{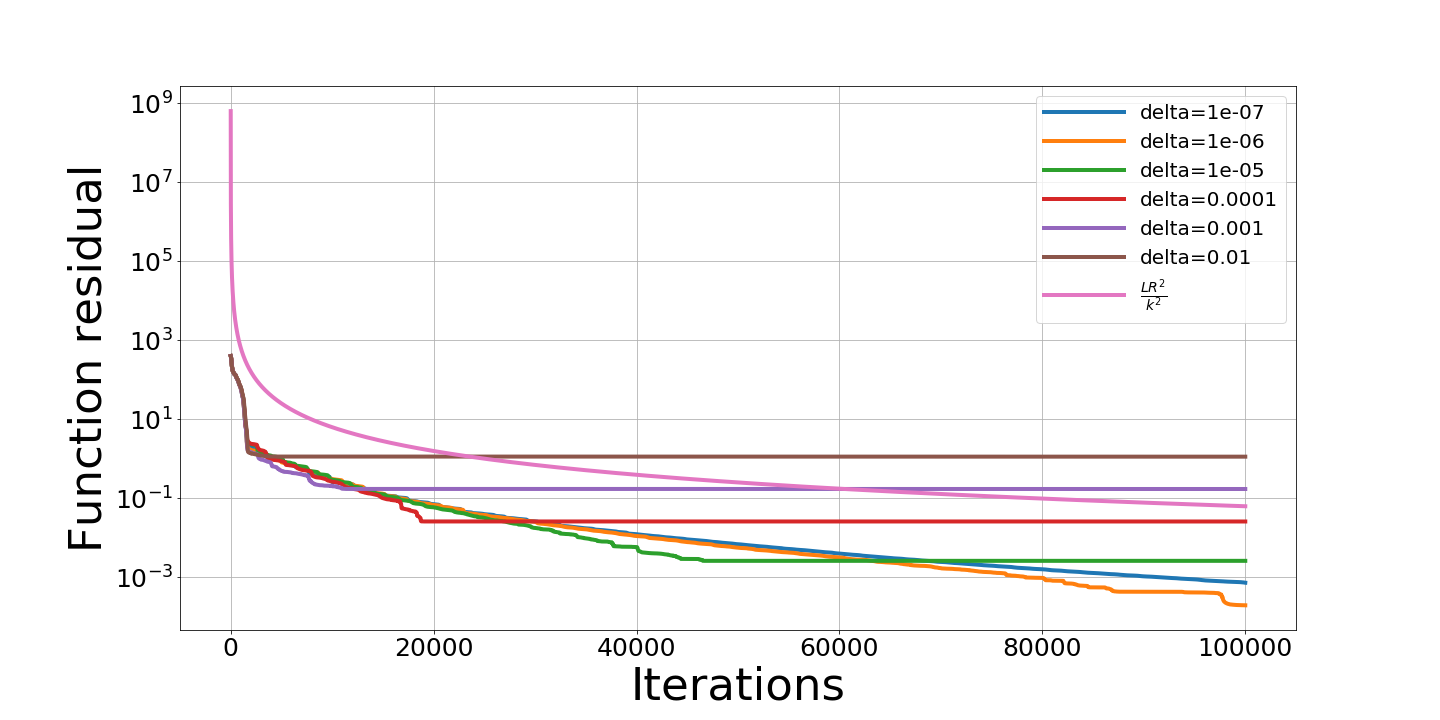} \label{fig:argmin_exp_f_N} }  
\hspace{4ex}
\subfigure{
\includegraphics[width=0.45\linewidth]{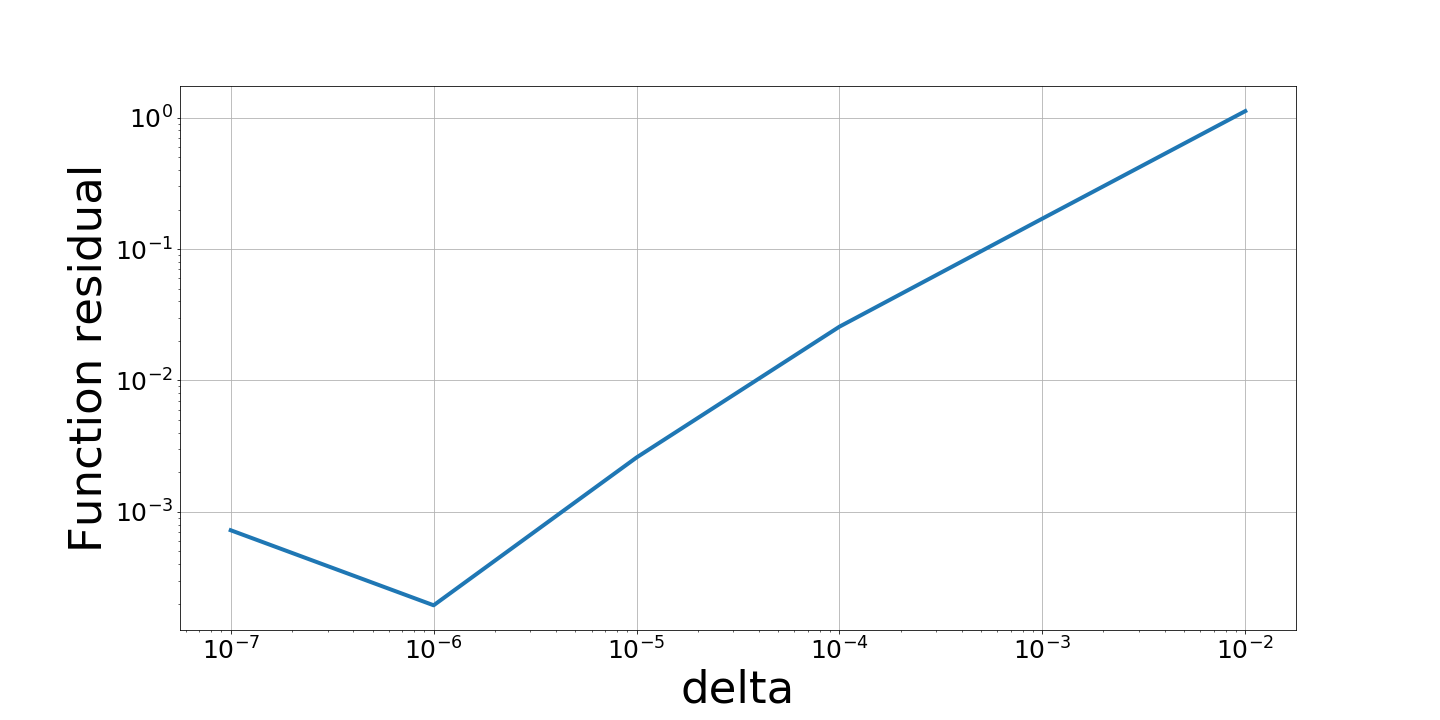} \label{fig:argmin_exp_f_delta} }
\caption{The dependencies of convergence on the inexactness of the subspace optimization problem in Algorithm \ref{alg:sesop}: \subref{fig:argmin_exp_f_N} convergence for different $\delta_4$; \subref{fig:argmin_exp_f_delta}  minimal values found for different $\delta_4$.} \label{fig:res_2}
\end{figure}

In the second experiment we studied the practical convergence  rate for different inexactness $\delta_j,j=\overline{2,4}$ when $\delta_1$ is fixed. In this experiments we take $\delta_1=10^{-3}$. Even in the ideal case, we cannot estimate the dependence of convergence on these parameters independently because when inexactness on the function of the subspace optimization problem \eqref{eqsubproblem} solution is small enough $\delta_4\rightarrow 0$, then other inexactness also tends to zero. We varied the inexactness of subspace optimization solution $\delta_4$. The results of the second experiment  are shown in Figure \ref{fig:res_2}.

In this case in Figure \ref{fig:argmin_exp_f_N} we can see that the convergence is significantly better than the theoretical estimation only for accuracy values $\delta=10^{-7},10^{-6},10^{-5}$. For values $10^{-2},10^{-3}$ there is no improvement after 20000 iterations and the theoretical estimation obtains better convergence. For value $10^{-4}$ the convergence stopped after 20000 iterations too but the theoretical convergence is not better due to a small number of iterations. In the figure \ref{fig:argmin_exp_f_delta} we can see that approached function value degrades with the linear rate depending on $\delta\geq10^{-6}$, which corresponds to the results of Theorem \ref{SESOP_theorem_full}. So, the proposed modification of the SESOP method is more sensitive to the accuracy of subproblem solution \eqref{eqsubproblem} than to the inexactness of the gradient.

Finally, we want to compare Algorithm \ref{alg:sesop} with an inexact gradient with another method that can work with gradient inexactness. We choice the known Similar Triangles Method (STM) with gradient inexactness from \cite{Vasin2021}. Similar to the previous experiment, we will consider two cases: the case of inexactness only in the gradient and the case of fixed additive gradient inexactness when subspace optimization is being solved inexactly too.

\begin{figure}[ht!]  
\vspace{-4ex} \centering \subfigure{
\includegraphics[width=0.45\linewidth]{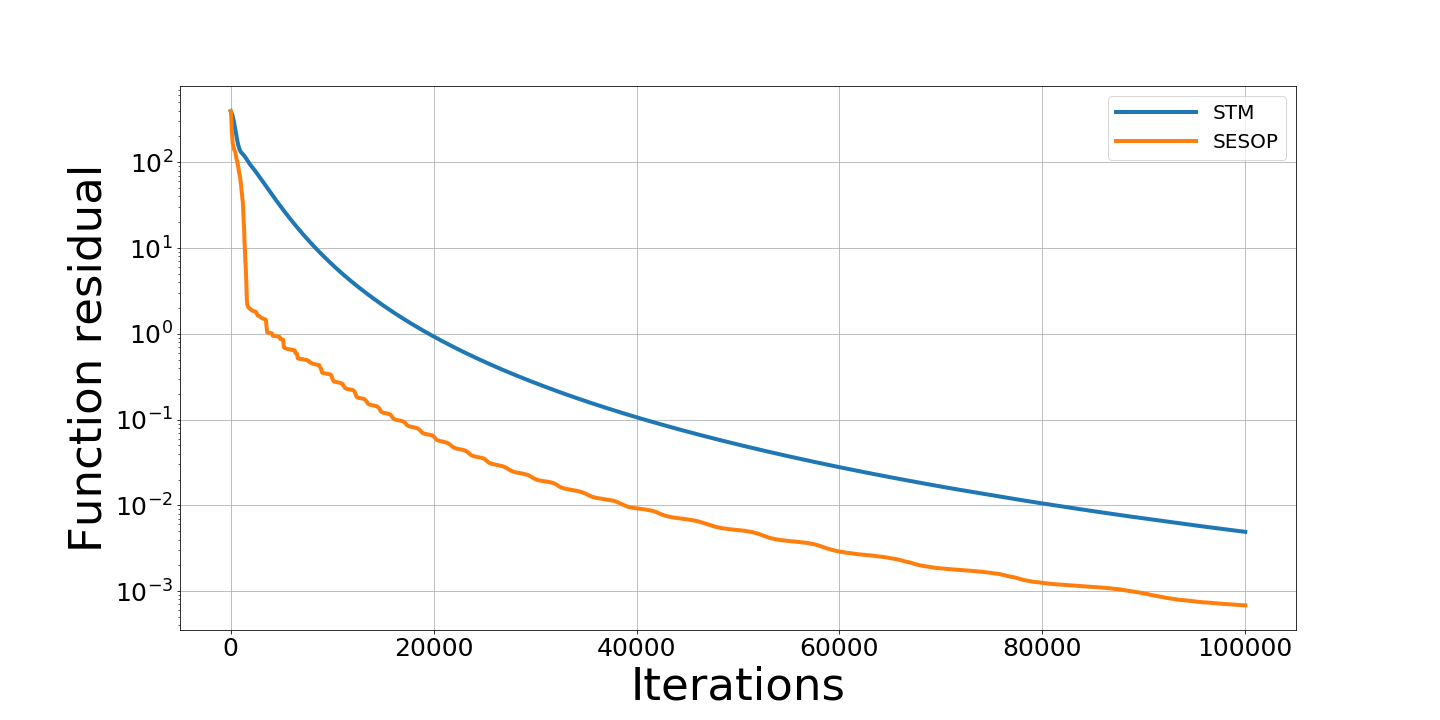} \label{fig:grad_SESOP_STM_1} }  
\hspace{4ex}
\subfigure{
\includegraphics[width=0.45\linewidth]{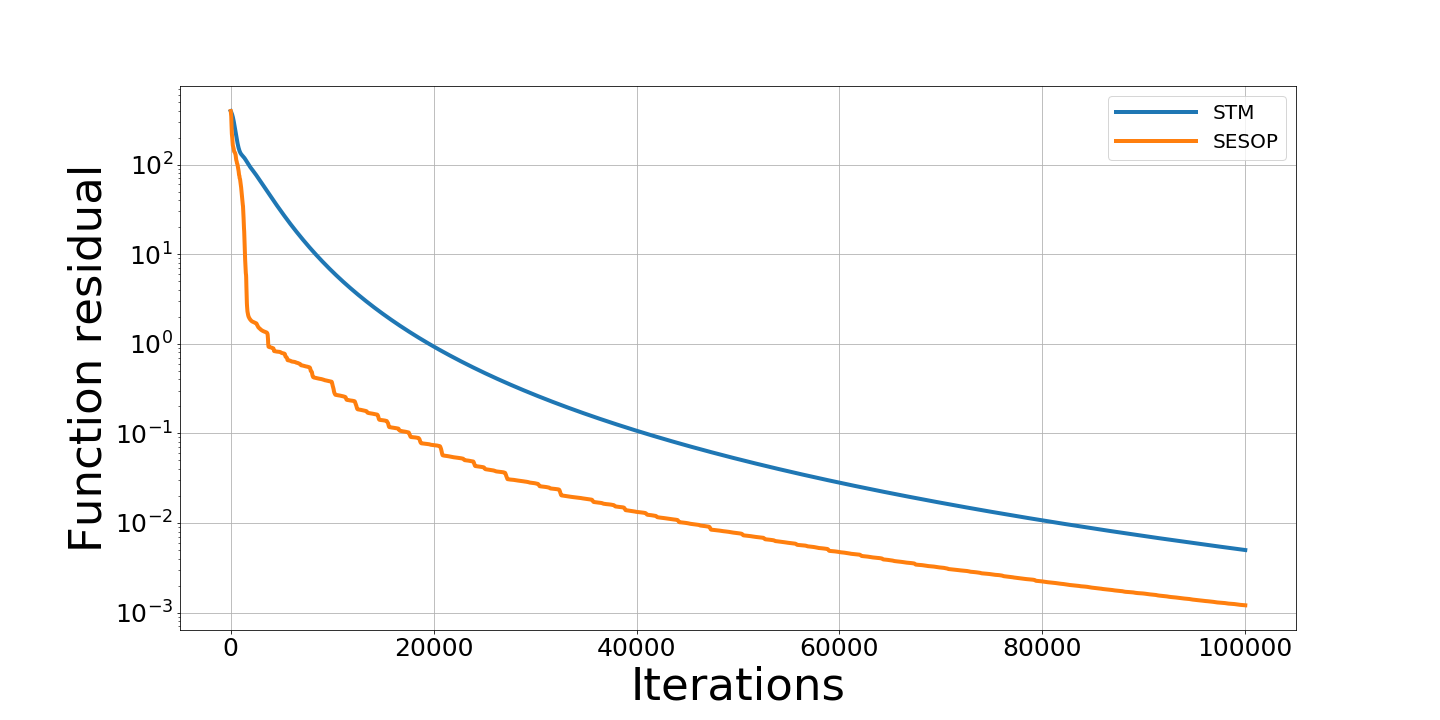} \label{fig:grad_SESOP_STM_2} }
\vspace{4ex}
\subfigure{
\includegraphics[width=0.45\linewidth]{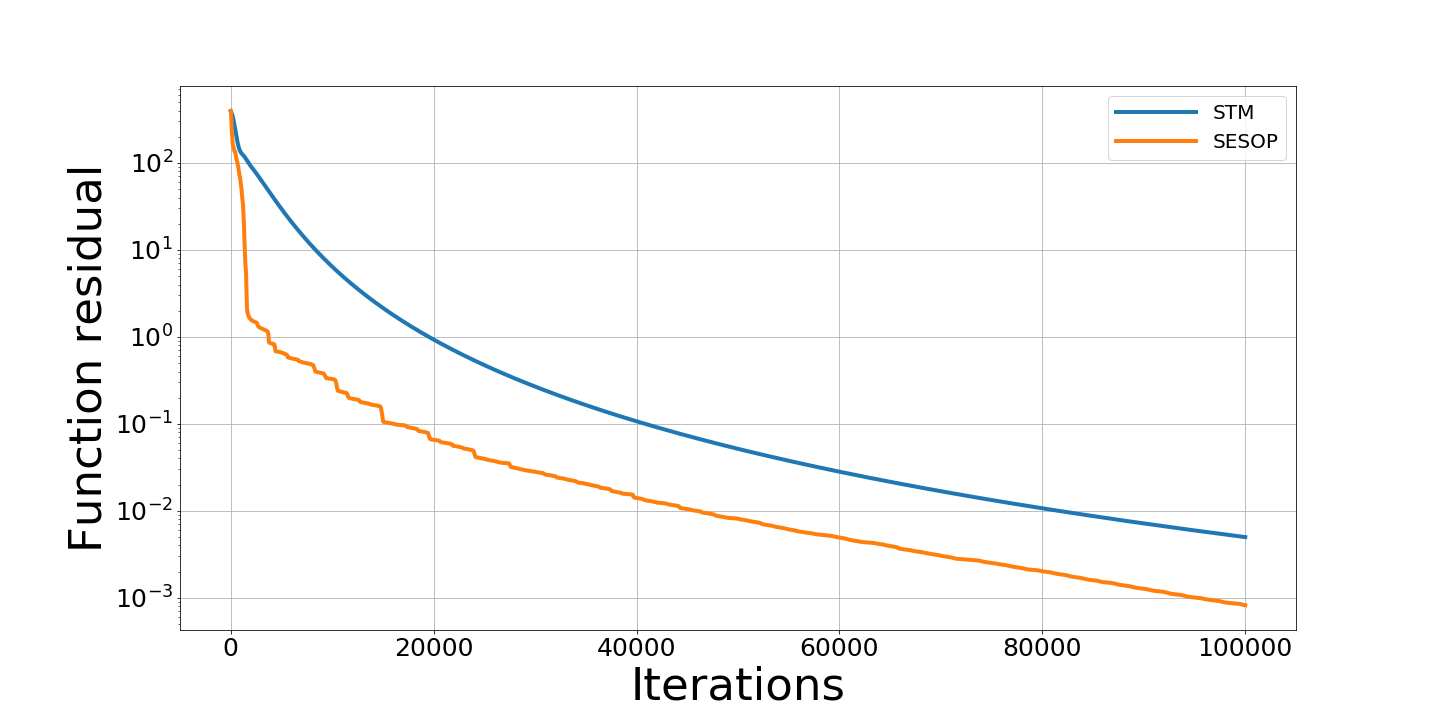} \label{fig:grad_SESOP_STM_3} }
\caption{The convergence of SESOP and STM for different additive noise with an exact solution of  the subspace optimization problem: \subref{fig:grad_SESOP_STM_1} $\delta_1=0.001$; \subref{fig:grad_SESOP_STM_2} $\delta_1=0.00001$; \subref{fig:grad_SESOP_STM_3} $\delta_1=0.1$.} \label{fig:res_3}
\end{figure}

The results for the first case for different values $\delta_1$ are presented in Figure \ref{fig:res_3}. We can see that because of the exact solution of the subspace optimization problem Algorithm \ref{alg:sesop} is almost everywhere better than the STM \cite{Vasin2021} with inexact gradient.

\begin{figure}[H]
    \centering
    \includegraphics[width=0.90\linewidth]{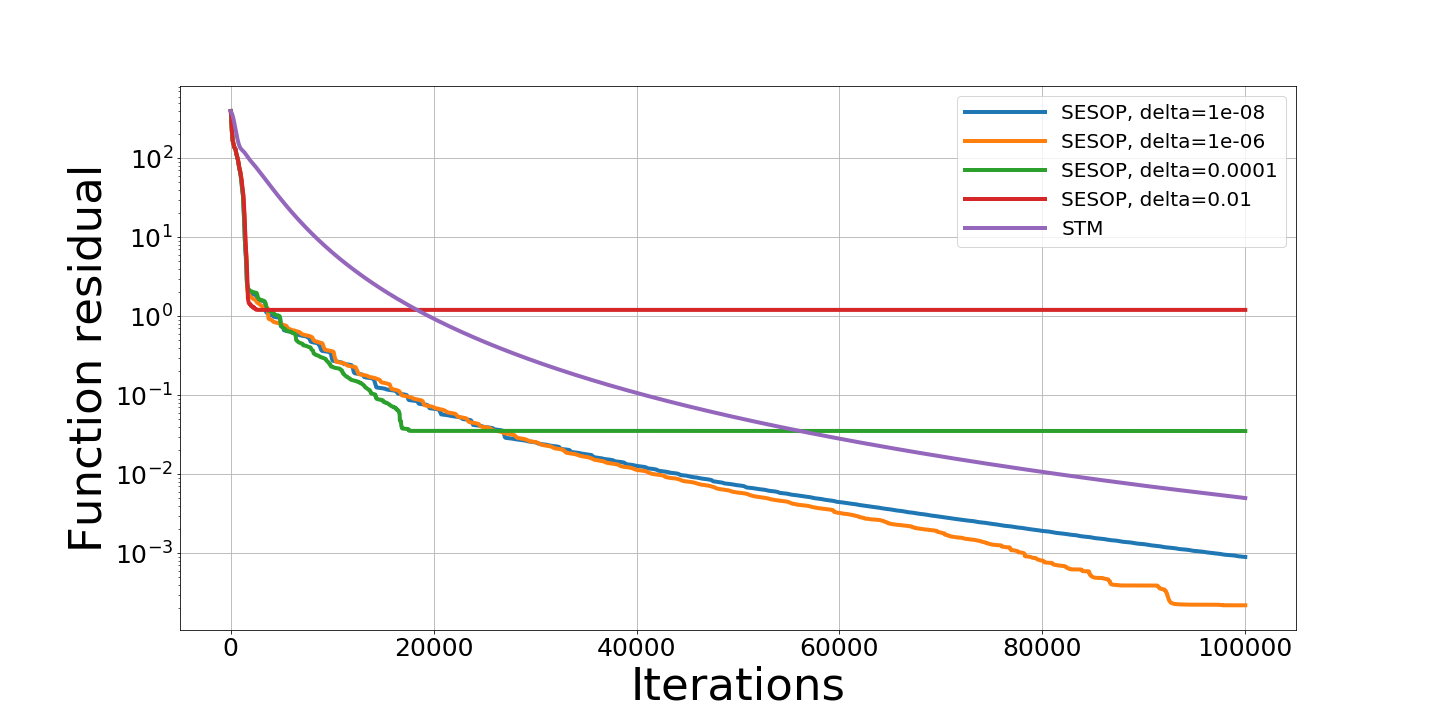}
    \caption{The convergence of Algorithm \ref{alg:sesop} with an inexact solution of the subspace optimization problem and STM for additive noise $\delta_1=10^{-3}$.}
    \label{fig:res_4}
\end{figure}

The results for the second case for different accuracy of the subspace problem solution are presented in Figure \ref{fig:res_4}. There is a natural result that for enough exact solution at each iteration, Algorithm \ref{alg:sesop} stays better than the STM. Nevertheless, for the inexactness in the low-dimensional subproblems solution larger or equal to $10^{-4}$ the STM becomes better than the provided method (Algorithm \ref{alg:sesop}).

\section*{Conclusion}

The contributions of the paper can be summarized as follows:

\begin{itemize}
	\item We propose one modification of the Sequential Subspace Optimization Method \cite{SESOP_2005} with a $\delta$-additive noise  in the gradient \eqref{InexactGrad}. For the first time, the result was obtained describing the influence of this inexactness on the estimate of the convergence rate, whereby the quantity $O(\delta \max_k \|x_k - x^*\|)$ is replaced by the constant $O(\delta \|x_0 - x^*\|)$, $\|x_0 - x^*\| \leq \max_k \|x_k - x^*\|$ .
	\item The influence of inexactness in solving auxiliary minimization problems \eqref{eqsubproblem} to the general theoretical estimate for Algorithm \eqref{alg:sesop} is investigated.
	\item We provide numerical experiments which demonstrate the effectiveness of the proposed approach in this paper. Algorithm \ref{alg:sesop} is compared with another known Similar Triangles Method (STM) with an additive gradient noise.

\end{itemize}

In the further works, we plan to continue the analysis of error accumulation in other methods for non-convex $\gamma$-quasar-convex functions. It is planned to develop some methods with auxiliary subproblems of dimension less than 3. In particular, we are going to consider the Conjugate Gradients method considered in \cite{NemirovskyYudin,AccWQC} and near-optimal methods from work \cite{NearOpt}. 

The authors are grateful to Alexander Gasnikov and Mohammad Alkousa for very useful discussions.


\newpage
\appendix

\section{Estimation for the sum of $w_k$}\label{AppendA}
We can prove the following estimations, which will be useful for the main results proved in this paper.

\begin{lemma}
\label{w_est_lemma}
Let $w_i$ be defined by the formula $\omega_0=1, \omega_i = \frac{1}{2}+\sqrt{\frac{1}{4}+\omega_{i-1}^2}.$

Then for any $T\in\mathbb{N}_0$ and for any $j, 0\leq j\leq T-1$, the following expressions hold:
\begin{equation}
    \label{w_k_0_T}
    \sum\limits_{k=0}^T w_k \leq \frac{1}{2}(T+2)(T+1),
\end{equation}
\begin{equation}
    \label{w_k_j_T}
    \sum\limits_{k=j+1}^T w_k \leq \frac{1}{2}T(T+3),
\end{equation}
\begin{equation}
    \label{j2_w_k}
    \sum\limits_{k=1}^T w_k k^2 \leq  \frac{7}{12}(T+1)^4,
\end{equation}
\begin{equation}
    \label{w_k2_T}
    \sum\limits_{k=0}^T w_k^2 \leq  \frac{1}{3}(T+1)^3.
\end{equation}
\end{lemma}

\begin{proof}
Using the upper bound \eqref{w_k_est} for $w_k$ we can estimate the left parts of \eqref{w_k_j_T} and \eqref{w_k_0_T} by the sum of arithmetic progression:
    $$\sum\limits_{k=j}^T w_k \leq \sum\limits_{k=j}^T (k+1) \leq \frac{1}{2}(T+j+2)(T-j+1).$$
For $j=0$ we have estimate \eqref{w_k_0_T}. Maximizing by $j$ on the segment $[1,T]$ we have the estimate \eqref{w_k_j_T}.

To get inequality \eqref{j2_w_k} we will use the estimation by integral of a monotonic function
$$    \sum\limits_{k=1}^T w_k k^2\leq \sum\limits_{k=1}^T (k+1)k^2 \leq \int\limits_{0}^{T+1} (k^3+k^2)dk = \frac{1}{4}(T+1)^4 + \frac{1}{3}(T+1)^3\leq \frac{7}{12}(T+1)^4.
$$
In a similar way we can obtain the last inequality
$$    \sum\limits_{k=0}^T w_k^2\leq \sum\limits_{k=0}^T (k+1)^2 \leq \int\limits_{0}^{T+1} k^2 dk = \frac{1}{3}(T+1)^3.
$$
\end{proof}

\section{Technical Lemma for Theorem \ref{SESOP_theorem_full}}

In this section, we propose a generalization of the proof of Lemma \ref{a_lemma_w} for the case when the additional problem was solved inexactly. In this case there is no orthogonality between $\nabla f(\mathbf{x}_T)$ and $\mathbf{d}_{T-1}^2$ was used in the proof of Lemma \ref{a_lemma_w}. Additionally, in this proof we made more accurate work with constants which gave more accurate estimates in the Theorem \ref{SESOP_theorem_full}.

\begin{lemma}
\label{a_lemma_w_full}
Let the inexact gradient meet to condition \eqref{grad_orth_i}.
Let $\{\mathbf{x}_j\}_j$ be a sequence of points generated by Algorithm \ref{alg:sesop} with conditions from Theorem \ref{SESOP_theorem_full}. Then the following condition is met:

\begin{equation}
    \left\|\sum\limits_{k=0}^{T} \omega_k g(\mathbf{x}_k)\right\|^2 \leq  2 \sum\limits_{j=0}^T \omega_j^2 \|g(\mathbf{x}_j)\|^2 + 2\left(5T^4 +21T^3 + 17T^2\right)\delta_1^2 +\frac{13}{6}(T+2)^4 \delta_2
\end{equation}
for all $T\geq 1$.
\end{lemma}

\begin{proof}
Let us define $W_{T} = \left\|\sum\limits_{k=0}^{T} \omega_k g(\mathbf{x}_k)\right\|$,  $\mathbf{d}^2_j=\sum\limits_{k=0}^{j} \omega_k g(\mathbf{x}_k)$. Note that we have the following equality for this value:
$$ W_T^2 = \omega_T^2 \|g(\mathbf{x}_T)\|^2 + 2\left\langle w_T g(\mathbf{x}_T), \mathbf{d}^2_{T-1}\right\rangle + W_{T-1}^2.$$

Obviously, $$\frac{d}{d\tau_3} f(x_{j-1}+\tau D_{j-1}) = \left\langle\nabla f(\mathbf{x}_j),  \mathbf{d}^2_{j-1}\right\rangle.$$ From \eqref{grad_orth_i} we have the following correlations:
\begin{equation*}
    |\left\langle g(\mathbf{x}_j), \ \mathbf{d}^2_{j-1}\right\rangle| = |\left\langle 
    \nabla f(\mathbf{x}_j),  \mathbf{d}^2_{j-1}\right\rangle| + |\left\langle g(\mathbf{x}_j) - \nabla f(\mathbf{x}_j),  \mathbf{d}^2_{j-1}\right\rangle|
\end{equation*}
and
\begin{equation*}
    |\left\langle g(\mathbf{x}_j), \ \mathbf{d}^2_{j-1}\right\rangle|  \leq j^2\delta_2 + \delta_1 W_{j-1}
\end{equation*}
for all $j\geq 1.$ So we have the following upper estimate for $W_T$:
\begin{equation}
        \label{up_W_est_full}
    \begin{aligned}
     W_T^2 &= \sum\limits_{j=0}^T \omega_j^2 \|g(\mathbf{x}_j)\|^2 + 2\sum\limits_{j=1}^T w_j\left\langle  g(\mathbf{x}_j), \mathbf{d}^2_{j-1}\right\rangle    \\
&\leq\sum\limits_{j=0}^T \omega_j^2 \|g(\mathbf{x}_j)\|^2 +2\delta_1 \sum\limits_{j=1}^T w_j W_{j-1} +  \delta_2 \sum\limits_{j=1}^T w_j j^2.   
    \end{aligned}
\end{equation}
On the other hand, we have the following inequality similar to the estimate from the proof of Lemma \ref{a_lemma_w} ($\sqrt{a + b} \leq \sqrt{a} + \sqrt{b}$ for each $a, b \geq 0$):
$$ W_j^2 \geq - j^2 \delta_2 - 2w_j \delta_1 W_{j-1} + W_{j-1}^2,$$
$$W_{j-1} \leq w_j\delta_1 + \sqrt{w_j^2 \delta_1^2 + W_j^2 + j^2 \delta_2} \leq 2w_j\delta_1 + j\sqrt{\delta_2} + W_j.$$

By induction we have the following estimate:
$$W_j \leq 2\delta_1 \sum\limits_{k=j+1}^T w_k + \sqrt{\delta_2}\sum\limits_{k=j+1}^T k + W_T\text{ for } j=\overline{0, T-1}.$$

From \eqref{w_k_j_T} we get
$$W_j \leq \delta_1 T(T+3) + \sqrt{\delta_2}(T-j)(T+j+1) + W_T\text{ for } j=\overline{0, T-1}.$$

Maximizing by $j$ the right side of the previous inequality, we get

$$W_j \leq \delta_1 T(T+3) + \sqrt{\delta_2} T(T+1) + W_T\text{ for } j=\overline{0, T-1}.$$

Let us use this inequality for \eqref{up_W_est_full} and get 
$$
W_T^2 \leq \sum\limits_{j=0}^T \omega_j^2 \|g(\mathbf{x}_j)\|^2 +\delta_1^2 C_1(T) +  2\delta_1W_T  C_2(T)  + \delta_2 C_3(T)  + 2\delta_1 \sqrt{\delta_2} C_4(T),
$$
where $C_1(T) = 2 T(T+3) \sum\limits_{j=1}^T w_j; C_2(T)=\sum\limits_{j=1}^T w_j; C_3(T)= \sum\limits_{j=1}^T w_j j^2; C_4(T)=T(T+1)\sum\limits_{j=1}^T w_j$ are functions of $T$ which will be estimated in the next steps.

Finally, we have the following inequality:

\begin{equation}
\label{quad_est_W}
    W_T^2 \leq  \sum\limits_{j=0}^T \omega_j^2 \|g(\mathbf{x}_j)\|^2 +  2\delta_1W_T  C_2(T) + \Delta,
\end{equation}
where $\Delta = \delta_1^2 \left(C_1(T)+C_4(T)\right) + \delta_2(C_3(T)+  C_4(T))$.
Solving the quadratic inequality \eqref{quad_est_W} we get the following estimate:
$$
    W_T \leq \delta_1  C_2(T) + \sqrt{\left(C_2(T)\right)^2 \delta_1^2 + \sum\limits_{j=0}^T \omega_j^2 \|g(\mathbf{x}_j)\|^2 +\Delta}.$$
Thus, we have
$$W_T^2 \leq 2 \Delta + 2 \sum\limits_{j=0}^T \omega_j^2 \|g(\mathbf{x}_j)\|^2 + 4 \left(C_2(T)\right)^2 \delta_1^2$$
and therefore
\begin{equation}
    \begin{aligned}
    W_T^2 \leq&\\ &2 \sum\limits_{j=0}^T \omega_j^2 \|g(\mathbf{x}_j)\|^2 +\\ &\left(4 \left(C_2(T)\right)^2   + 2C_1(T)+2C_4(T)\right)\delta_1^2 +\\& \left(C_3(T)+C_4(T)\right)\delta_2.
\end{aligned}
\end{equation}

When we use the definition of functions $C_j(T),j=\overline{1,4}$ and estimations \eqref{w_k_j_T}, \eqref{j2_w_k}, we have

$$W_T^2 \leq 2 \sum\limits_{j=0}^T \omega_j^2 \|g(\mathbf{x}_j)\|^2 + 2\left(5T^4 +28 T^3 + 39 T^2\right)\delta_1^2 +\frac{13}{6} (T+3)^4 \delta_2.$$

\end{proof}

\section{Proof for Theorem \ref{SESOP_theorem_full}}

\begin{proof}
By the constructing of the $\mathbf{x}_{k+1}$ we have the following inequality:
\begin{equation}
    \label{k_min_ineq_full}
    f(\mathbf{x}_{k+1}) \leq \min_{\mathbf{s}\in\mathbb{R}^3} f\left(\mathbf{x}_k + \sum\limits_{i=0}^2 s_i d_k^i\right) + \delta_4 \leq f\left(\mathbf{x}_k + s_0 g(\mathbf{x}_k)\right)  + \delta_4.
\end{equation}

On the other hand, we have 

$$f(\mathbf{y})\leq f(\mathbf{x}) + \langle\nabla f(\mathbf{x}), \mathbf{y} - \mathbf{x}\rangle + \frac{L}{2}\|\mathbf{x}-\mathbf{y}\|^2$$
for all $\mathbf{x}, \mathbf{y}\in\mathbb{R}^n$.

From condition $\eqref{g_cond_i}$ we can obtain the corresponding inequality for the inexact gradient:

$$f(\mathbf{y})\leq f(\mathbf{x}) + \langle g(\mathbf{x}), \mathbf{y} - \mathbf{x}\rangle + \frac{L}{2}\|x-y\|^2 + \delta_1 \|\mathbf{x}-\mathbf{y}\|.$$

Using the inequality $\delta_1 \|\mathbf{x}-\mathbf{y}\| = \left(\sqrt{\frac{1}{L}}\delta_1\right) \left(\sqrt{\frac{L}{1}}\|\mathbf{x}-\mathbf{y}\|\right)\leq \frac{\delta_1^2}{2L} + \frac{L}{2}\|\mathbf{x}-\mathbf{y}\|^2$ we have

\begin{equation}
    \label{g_L_ineq_full}
f(\mathbf{y})\leq f(\mathbf{x}) + \langle g(\mathbf{x}), \mathbf{y} - \mathbf{x}\rangle + L\|x-y\|^2 + \frac{\delta_1^2}{2L} .
\end{equation}

Using this inequality for the right part of \eqref{k_min_ineq_full} for $\mathbf{y}:=\mathbf{x}_k + s_0 g(\mathbf{x}_k), \mathbf{x}=\mathbf{x}_k$ we have
\begin{equation*}
    f(\mathbf{x}_{k+1}) \leq f(\mathbf{x}_k) +\left(s_0 + s_0^2 L\right)\|g(\mathbf{x}_k)\|^2 + \frac{1}{2L}\delta_1^2   + \delta_4 
\end{equation*}
for any $s_0\in\mathbb{R}$. To simplify this proof we define $\tilde{
\delta} = \delta_1^2   + 2L \delta_4 $, and the last inequation in this case will be rewritten as

\begin{equation*}
    f(\mathbf{x}_{k+1}) \leq f(\mathbf{x}_k) +\left(s_0 + s_0^2 L\right)\|g(\mathbf{x}_k)\|^2 + \frac{1}{2L} \tilde{\delta}
\end{equation*}
and
\begin{equation*}
     -\left(s_0 + s_0^2 L\right) \|g(\mathbf{x}_k)\|^2 \leq f(\mathbf{x}_k) - f(\mathbf{x}_{k+1}) + \frac{1}{2L} \tilde{\delta}.
\end{equation*}

Maximizing the left part on $s_0$ we have the following estimate for inexact gradient norm:

\begin{equation}
      \|g(\mathbf{x}_k)\|^2 \leq 4L(f(\mathbf{x}_k) - f(\mathbf{x}_{k+1})) + 2\tilde{\delta}.
\end{equation}

We have that $|\left\langle \nabla f(\mathbf{x}_k), \mathbf{x}_k - \mathbf{x}_0\right\rangle| \leq \delta_3$ for all $k>0$, according to Theorem \ref{SESOP_theorem_full} conditions. Because of it we can write

$$\langle \nabla f(\mathbf{x}_k), \mathbf{x}_k-\mathbf{x}^*\rangle \leq \langle \nabla f(\mathbf{x}_k), \mathbf{x}_0-\mathbf{x}^* \rangle +\delta_3$$

 From this and the definition of inexact gradient we have

$$f(\mathbf{x}_k)-f(\mathbf{x}^*)\leq \frac{1}{\gamma} \langle g(\mathbf{x}_k), \mathbf{x}_0-\mathbf{x}^* \rangle + \delta_1 \frac{R}{\gamma} + \delta_3$$
Further we define $\overline{\delta} = \delta_1 \frac{R}{\gamma}+\delta_3$.

Similarly to the proof of Theorem 3.1 from \cite{AccWQC} we have the following chain of statements:

\begin{equation}
\label{eps_1_full}
    \begin{aligned}
        \sum\limits_{k=0}^{T-1}\omega_k(f(\mathbf{x}_k)-f^*)&\leq \frac{1}{\gamma} \left\langle  \sum\limits_{k=0}^{T-1} \omega_k g(\mathbf{x}_k), \mathbf{x}_0-\mathbf{x}^* \right\rangle +\overline{\delta} \sum\limits_{k=0}^{T-1} \omega_k\\
        &\leq\frac{1}{\gamma} \left\|\sum\limits_{k=0}^{T-1} \omega_k g(\mathbf{x}_k)\right\|R +\overline{\delta} \sum\limits_{k=0}^{T-1} \omega_k.
    \end{aligned}
\end{equation}

According to Lemma \ref{a_lemma_w_full}  and \eqref{w_k_est} we can estimate the first multiplier

\begin{align*}
    \left\|\sum\limits_{k=0}^{T-1} \omega_k g(\mathbf{x}_k)\right\|^2 \leq \\ 2 \sum\limits_{j=0}^T \omega_j^2 \|g(\mathbf{x}_j)\|^2 + 2\left(5T^4 +28 T^3 + 39 T^2\right)\delta_1^2 +\frac{13}{6} (T+3)^4 \delta_2  \leq \\
    8L \sum\limits_{k=0}^{T-1} \omega_k^2  (f(\mathbf{x}_k)-f(\mathbf{x}_{k+1})) + \Delta,
\end{align*}
where $\Delta = 2\left(5T^4 +28 T^3 + 39 T^2\right)\delta_1^2 +\frac{13}{6} (T+3)^4 \delta_2  + \frac{4}{3}(T+1)^3\tilde{\delta}.$ The coefficient before $\tilde{\delta}$ is obtained from estimation \eqref{w_k2_T}.

Note that the choice of $\omega_k$ is equivalent to choosing the greatest $\omega_k$ satisfying
$$\omega_k = \begin{cases}1, \text{if $k=0$},\\ \omega_k^2-\omega_{k-1}^2, \text{ptherwise}.\end{cases}$$

Now we can estimate the left part of \eqref{eps_1_full} denoting $\varepsilon_k=f(\mathbf{x}_k)-f^*$ by the following way:

\begin{equation*}
    \begin{aligned}
        S &= \sum\limits_{k=0}^{T-1}\omega_k \varepsilon_k  \\
        &\leq\frac{1}{\gamma} \left\|\sum\limits_{k=0}^{T-1} \omega_k g(\mathbf{x}_k)\right\|R +\overline{\delta} \sum\limits_{k=0}^{T-1} \omega_k\\
         & \leq\left(\frac{8LR^2}{\gamma^2} \sum\limits_{k=0}^{T-1} \omega_k^2 (\varepsilon_k-\varepsilon_{k+1}) + \Delta \right)^{\frac{1}{2}} +\overline{\delta} \sum\limits_{k=0}^{T-1} \omega_k\\
      &\leq\left(\frac{8LR^2}{\gamma^2} \sum\limits_{k=0}^{T-1} \omega_k^2 (\varepsilon_k-\varepsilon_{k+1})\right)^{\frac{1}{2}} + \sqrt{\Delta} +\overline{\delta} \sum\limits_{k=0}^{T-1} \omega_k\\
         &=\sqrt{\frac{8LR^2}{\gamma^2}} \sqrt{S - \varepsilon_T\omega_{T-1}^2} +  \sqrt{\Delta}  +\overline{\delta} \sum\limits_{k=0}^{T-1} \omega_k.
    \end{aligned}
\end{equation*}

From the inequality above we have

\begin{equation*}
    \omega_{T-1}^2 \varepsilon_T \leq S - \frac{\gamma^2}{8LR^2}\left(S - \overline{\delta} \sum\limits_{k=0}^{T-1} \omega_k - \sqrt{\Delta} \right)^2.
\end{equation*}

Maximizing the right part by $S$ we get

\begin{equation*}
    \omega_{T-1}^2 \varepsilon_T \leq \frac{2LR^2}{\gamma^2} +  \overline{\delta} \sum\limits_{k=0}^{T-1} \omega_k + \sqrt{\Delta}.
\end{equation*}

From \eqref{w_k_est} we get

\begin{equation}
    \begin{aligned}
    \omega_{T-1}^2 \varepsilon_T &\leq \frac{2LR^2}{\gamma^2} +  \overline{\delta} \sum\limits_{k=0}^{T-1} \omega_k +  \sqrt{\Delta}\\
    &\leq\frac{2LR^2}{\gamma^2} +  \frac{1}{2}T(T+1) \overline{\delta} + \sqrt{\Delta}.
    \end{aligned}
\end{equation}

Dividing the both parts of this inequality by $w_{T-1}^2 \leq T^2$ for $T> 0$ and using the lower estimate \eqref{w_k_est} for it we get

\begin{equation*}
    \begin{aligned}
 \varepsilon_T \leq \frac{8LR^2}{\gamma^2 T^2} + \frac{1}{2}\left(1+\frac{1}{T}\right)\overline{\delta} + \frac{1}{T^2}\sqrt{\Delta}
    \end{aligned}
\end{equation*}

By definition of $\Delta$ we have:

\begin{equation}
\label{Delta}
    \frac{1}{T^2}\sqrt{\Delta} =
    \left(4+\frac{13}{\sqrt{T}}\right)\delta_1 + 2 \left(1+\frac{3}{T}\right)^2\sqrt{\delta_2} + 5\sqrt{\frac{L\delta_4}{T}}.
\end{equation}

To get the coefficient before $\delta_4$ we use $\frac{8}{3}(T+1)^3\leq \frac{64}{3}T^3 < 25T^3$ for $T \geq 1$. Using the estimation \eqref{Delta} and the definition of $\overline{\delta}$ we have an estimate of the error of our algorithm

\begin{equation}
    \begin{aligned}
 \varepsilon_T \leq \frac{8LR^2}{\gamma^2 T^2} + A(T)\delta_1 + B(T)\sqrt{\delta_2} +  + \frac{1}{2}\left(1+\frac{1}{T}\right) \delta_3 + 5\sqrt{\frac{L\delta_4}{T}},
    \end{aligned}
\end{equation}
where $A(T) = \frac{1}{2}\left(1+\frac{1}{T}\right)\frac{R}{\gamma} + \left(4+\frac{13}{\sqrt{T}}\right);$ $B(T) = 2 \left(1+\frac{3}{T}\right)^2 $.

Obviously, for $T\geq 8$ we can write the following estimate:

\begin{equation}
    \begin{aligned}
 \varepsilon_T \leq \frac{8LR^2}{\gamma^2 T^2} + \left(\frac{R}{\gamma}+ 10\right)\delta_1 + 4\sqrt{\delta_2} +  \delta_3 + 5\sqrt{\frac{L\delta_4}{T}},
    \end{aligned}
\end{equation}
Q.E.D.
\end{proof}

\end{document}